\documentclass{amsart}
\usepackage{graphicx}

\usepackage{amsmath}
\usepackage{mathtools}
\usepackage
{amstext, amscd, setspace, tikz-cd,mathtools, enumerate, graphics, latexsym, siunitx}

\usepackage{pst-node}
\usepackage{tikz-cd}

\usepackage{amssymb}
\usepackage{hyperref}
\hypersetup{colorlinks=true,
    linkcolor=blue,
    filecolor=magenta,      
    urlcolor=cyan,}

\usepackage{PTSerif}
\usepackage{graphicx}

\usepackage[cmtip,all]{xy}
\newcommand{\properideal}{%
\mathrel{\ooalign{$\lneq$\cr\raise.22ex\hbox{$\lhd$}\cr}}}

\newcommand{\p}{\mathfrak{p}}
\newcommand{\D}{\mathcal{D}}
\newcommand{\m}{\mathfrak{m}}

\newcommand{\F}{\mathcal{F}}
\newcommand{\Z}{\mathbb{Z}}

\newcommand{\injdim}{\operatorname{injdim}}

\newcommand{\Ass}{\operatorname{Ass}}

\newcommand{\Hom}{\operatorname{Hom}}
\newcommand{\Mod}{\operatorname{Mod}}

\newcommand{\Ext}{\operatorname{Ext}}

\theoremstyle{plain}
\newtheorem{theorem}{Theorem}[section]
\newtheorem{lemma}[theorem]{Lemma}

\newtheorem{proposition}[theorem]{Proposition}
\newtheorem{remark}[theorem]{Remark}
\newtheorem{definition}[theorem]{Definition}

\newtheorem{example}[theorem]{Example}

\newtheorem{setup}[theorem]{Setup}

\title[On the structure theorem of graded components of $\F$-finite, $\F$-module]{On  the structure theorem of graded components of $\F$-finite, $\F$-modules over certain polynomial ring}

\author{\textsc{Sayed Sadiqul Islam}}
\address{Department of Mathematics, IIT Bombay, Powai, Mumbai 400076, India}
\email{ssislam1997@gmail.com, 22d0786@iitb.ac.in}

\date{\today}

\begin{document}

\begin{abstract}
Let $K$ be a field of characteristic $p>0$, $A=K[[Y]]$ be a power series ring in one variable and $Q(A)$ be the field of fraction of $A$. Suppose that $R=A[X_1,\ldots,X_n]$ is a standard $\mathbb{N}^n$-graded polynomial ring over $A$, i.e., $\deg (A)=\underline{0}\in \mathbb{N}^n$ and $\deg(X_j)=e_j\in \mathbb{N}^n$. Assume that $M=\bigoplus_{\underline{u}\in \mathbb{Z}^n} M_{\underline{u}}$ is a $\mathbb{Z}^n$-graded $\F$-finite, $\F$-module over $R$. In this article we prove that,
    $\displaystyle M_{\underline{u}}\cong E(A/YA)^{a(\underline{u})}\oplus Q(A)^{b(\underline{u})}\oplus A^{c(\underline{u})}$
    for some finite numbers $a(\underline{u}), b(\underline{u}), c(\underline{u})\geq 0$. Let  for a subset of $U$ of $\mathcal{S}=\{1, \ldots, n\}$, define a block to be the set
$\displaystyle\mathcal{B}(U)=\{\underline{u} \in \Z^n \mid u_i \geq 0 \mbox{ if } i \in U \mbox{ and } u_i \leq -1 \mbox{ if } i \notin U \}$. Note that $\bigcup_{U\subseteq \mathcal{S}}\mathcal{B}(U)=\mathbb{Z}^n$. We  prove that the sets  $\{a(\underline{u})\mid \underline{u}\in \mathbb{Z}^n\}$, $\{b(\underline{u})\mid \underline{u}\in \mathbb{Z}^n\}$ and $\{c(\underline{u})\mid \underline{u}\in \mathbb{Z}^n\}$ are constant on $\mathcal{B}(U)$ for each subset $U$ of $\{1,\ldots,n\}$. In particular, these results holds for composition of local cohomology modules of the form $ H^{i_1}_{I_1}(H^{i_2}_{I_2}(\dots H^{i_r}_{I_r}(R)\dots)$ where $I_1,\ldots,I_r$  are $\mathbb{N}^n$-graded ideals of $R$. This provides a positive characteristic analogue of the results proved in \cite{TS-23} by the authors in characteristic zero.
\end{abstract}

\maketitle
\section{Introduction}
Since its introduction by Grothendieck in the early 1960s, local cohomology has become an indispensable tool and a central object of study in commutative algebra and algebraic geometry. Local cohomology modules are not finitely generated in general. The lack of finite generation has turned attention of researchers to study finiteness properties of local cohomology modules; see, for instance, \cite{HS-93, Lyu-93, Lyu-97, Lyu-2000}. The theories of $\D$-modules and $\F$-modules, developed by Lyubeznik in his seminal papers \cite{Lyu-93, Lyu-97} have been very useful in the study of finiteness properties of local cohomology modules in characteristic zero as well as in positive characteristic.

Let $S=\oplus_{n\geq 0}S_n$ be a standard graded Noetherian ring and $S_+$ be irrelevant ideal. It is well known  \cite[Theorem 15.1.5]{BS-13} that the graded componets  $H^i_{S_+}(M)_n$ is finitely generated over $S_0$ and $H^i_{S_+}(M)_n=0$ for all $n\gg 0$. This fact motivated Puthenpurakal to study the graded components of $H^i_I(R)$ where  $R$ is a standard graded polynomial ring $A[X_1,\ldots,X_n]$ over a regular ring  $A$  containing a field of characteristic zero and $I$ is an arbitrary homogeneous ideal of $R$ (see, \cite{TP-collect}). Later, in a series of papers \cite{TS-19,TS-24, TP-25, TP-koszul} the authors carried out a comprehensive study of graded components of local cohomology modules under various setups. The asymptotic stability of invariants associated to the bigraded components of $H^i_I(R)$ where $R=A[X_1,\ldots,X_m,Y_1,\ldots,Y_n]$ is a standard bigraded polynomial ring over a regular ring $A$ containing a field and $I$ is a bi-homogeneous ideal of $R$ are studied in \cite{BPRS-26, ST-25} depending on the characteristic of the field.

Let $A$ be a Dedekind domain of characteristic zero such that its localization at every
maximal ideal has mixed characteristic with finite residue field. Suppose that $R=A[X_1,\ldots,X_n]$ is a standard $\mathbb{N}^n$-graded polynomial ring over $A$, i.e., $\deg (A)=\underline{0}\in \mathbb{N}^n$ and $\deg(X_j)=e_j\in \mathbb{N}^n$. Let $I$ be a $\mathfrak{C}$-monomial ideal of $R$. The authors have  produced a structure theorem for the multigraded components of the local cohomology modules $M=H^i_I(R)$ for $i\geq 0$ (see, \cite[Theorem 1.3]{TS-25}). It is also proved that for a fixed $\underline{u}\in\mathbb{Z}^n$, the Bass numbers $\mu_i(\p,M_{\underline{u}})$ are finite for each prime ideal $\p$ in $A$ and for every $i\geq 0$ (see, \cite[Corollary 6.6]{TS-25}). Let  for a subset of $U$ of $\mathcal{S}=\{1, \ldots, n\}$, define a block to be the set
$\displaystyle\mathcal{B}(U)=\{\underline{u} \in \Z^n \mid u_i \geq 0 \mbox{ if } i \in U \mbox{ and } u_i \leq -1 \mbox{ if } i \notin U \}$. Note that $\bigcup_{U\subseteq \mathcal{S}}\mathcal{B}(U)=\mathbb{Z}^n$. Recently, the authors in \cite{ST-26} proved that for a fixed prime ideal $\p$ in $A$ and $i\geq 0$, the set of Bass numbers $\{\mu_i(\p,M_{\underline{u}})\mid \underline{u}\in \mathbb{Z}^n\}$ is constant 
on $\mathcal{B}(U)$ for each subset $U$ of $\{1, \ldots, n\}$ (see, \cite[Theorem 7.2]{ST-26}).

Let $K$ be a field of characteristic zero and let $R=A[X_1,\ldots,X_n]$ be a standard $\mathbb{N}^n$-graded polynomial ring over $A$ where $A=K[[Y]]$ is a formal power series ring in one variable. The authors proved an analogue of the structure theorem for components of $H^i_I(R)$ in this case (see, \cite[Theorem 7.1]{TS-23}). It is also proved that for a fixed prime ideal $\p$ in $A$ and $i\geq 0$, the set of Bass numbers $\{\mu_i(\p,M_{\underline{u}})\mid \underline{u}\in \mathbb{Z}^n\}$ is constant 
on $\mathcal{B}(U)$ for each subset $U$ of $\{1, \ldots, n\}$ (see, \cite[Theorem 5.6]{TS-23}).
\medskip

In this article, we mainly establish the results proved in \cite{TS-23} to the setting of positive characteristic. The techniques and results from \cite{TP-koszul} will be used extensively in our proofs. Now we discuss the results proved in this paper. Consider the following setup.
\begin{setup}\label{main setup intro}
  Let $A$ be a regular ring containing a field of characteristic $p>0$. Suppose that $R=A[X_1,\ldots,X_n]$ is a standard $\mathbb{N}^n$-graded polynomial ring over $A$, i.e., $\deg (A)=\underline{0}\in \mathbb{N}^n$ and $\deg(X_j)=e_j\in \mathbb{N}^n$.   
\end{setup}
Let the hypothesis be as in \ref{main setup intro}. The following result  proves that the vanishing of almost graded components of $M$ implies vanishing of $M$.
\medskip

\noindent
\textbf{Theorem A.} (Theorem \ref{proof of vanishing theorem})
\phantomsection \label{Theorem A}
\textit{Let $R$ be as in Setup \ref{main setup intro}. Assume that $M=\bigoplus_{\underline{u}\in \mathbb{Z}^n} M_{\underline{u}}$ is a $\mathbb{Z}^n$-graded $\F$-finite, $\F$-module over $R$. If $M_{\underline{u}}=0$ for all $|u_i|\gg 0 $, then $M=0$.
}
\medskip

Next, we prove rigidity results concerning the graded components of $M$.
\medskip

\noindent
\textbf{Theorem B.} (Theorem \ref{proof of rigidity `t`heorem})
\phantomsection \label{Theorem B}
\textit{Let $R$ be as in Setup \ref{main setup intro}. Assume that $M=\bigoplus_{\underline{u}\in \mathbb{Z}^n} M_{\underline{u}}$ is a $\mathbb{Z}^n$-graded $\F$-finite, $\F$-module over $R$. Then, the following are equivalent:
     \begin{enumerate}[\rm (i)]
         \item $M_{\underline{u}}\neq 0$ for all $\underline{u}\in \mathcal{B}(U)$;
         \item $M_{\underline{w}}\neq 0$ for some $\underline{w}\in \mathcal{B}(U)$.
     \end{enumerate}
}
\medskip
 In the next theorem, we show that the set of associated primes, Bass numbers,  injective dimensions, and the support dimensions of the graded components are constant on $\mathcal{B}(U)$ for each subset $U$ of $\{1,\ldots,n\}$.
 \medskip
 
\noindent
\textbf{Theorem C.} (Theorem \ref{bass numbers,injective dimension, associted primes are constant in block})
\phantomsection \label{Theorem C}
\textit{Assume the hypothesis as in  \ref{main setup intro}. Let $\mathcal{B}(U)$ denote a block corresponding to a subset $U$ of $\{1,\ldots,n\}$.
\begin{enumerate}[\rm (i)]
    \item Fix a  prime ideal $\p$ of $A$ and  $i\geq 0$. Then, the set of Bass numbers $\{\mu_i(\p,M_{\underline{u}})\mid \underline{u}\in \mathbb{Z}^n\}$ is constant 
on $\mathcal{B}(U)$.
\item The sets $\{\injdim M_{\underline{u}}\mid \underline{u}\in \mathbb{Z}^n\}$, $\{\dim M_{\underline{u}}\mid \underline{u}\in \mathbb{Z}^n\}$ and $\{\Ass_A M_{\underline{u}}\mid \underline{u}\in \mathbb{Z}^n\}$ are constant 
on $\mathcal{B}(U)$
\end{enumerate}
}
\medskip
We prove an
 structure theorem similar to mixed characteristic case (see, \cite[Theorem 1.3]{TS-25}) in positive characteristic where $A = K[[Y ]]$. When  $\operatorname{char} K=0$, the result is proved for $H^i_I(R)$ where $I$ is a $\mathfrak{C}$-monomial ideal of $R$ (\cite[Theorem 7.1]{TS-23}).  We note that unlike the mixed characteristic case, the torsion part in this case does not have any finitely generated summand. More precisely, we prove the following.
\medskip

\noindent
\textbf{Theorem D.} (Theorem \ref{structure theorem})
\phantomsection \label{Theorem D}
\textit{ Let $K$ be field of characteristic $p>0$, $A=K[[Y]]$ be a power series ring in one variable and $Q(A)$ be the field of fraction of $A$. Suppose that $R=A[X_1,\ldots,X_n]$ is a standard $\mathbb{N}^n$-graded polynomial ring over $A$, i.e., $\deg (A)=\underline{0}\in \mathbb{N}^n$ and $\deg(X_j)=e_j\in \mathbb{N}^n$. Assume that $M=\bigoplus_{\underline{u}\in \mathbb{Z}^n} M_{\underline{u}}$ is a $\mathbb{Z}^n$-graded $\F$-finite, $\F$-module over $R$. Then,
    $$M_{\underline{u}}\cong E(A/YA)^{a(\underline{u})}\oplus Q(A)^{b(\underline{u})}\oplus A^{c(\underline{u})}$$
    for finite numbers $a(\underline{u}), b(\underline{u}), c(\underline{u})\geq 0$.
}
\medskip

In the next theorem, we study the behaviour of the numbers appearing in the above structure theorem in each block $\mathcal{B}(U)$ corresponding to a subset $U$ of $\{1,\ldots,n\}$. Similar to characteristic zero case \cite[Remark 7.2]{TS-23}, we prove the following.

\medskip
\noindent
\textbf{Theorem E.} (Theorem \ref{numbers are constant})
\phantomsection \label{Theorem E}
\textit{ Let the hypothesis be as in  \hyperref[Theorem D]{Theorem D}. Then,the sets $\{a(\underline{u})\mid \underline{u}\in \mathbb{Z}^n\}$, $\{b(\underline{u})\mid \underline{u}\in \mathbb{Z}^n\}$ and $\{c(\underline{u})\mid \underline{u}\in \mathbb{Z}^n\}$ are constant 
on $\mathcal{B}(U)$ for each subset $U$ of $\{1, \ldots, n\}$.
}
\medskip

We now describe in brief the contents of this paper. The paper is organized in six sections. Section 2 introduces the necessary preliminaries and the  supporting results we need in this paper. In Section 4, we study graded components of $H^i_I(R)$ when $R=K[X_1,\ldots,X_n]$ is a standard $\mathbb{N}^n$-graded polynomial ring over a field $K$ of positive characteristic. In section 5, we generalize the results proved in Section 4 for more general setup, more precisely we prove \hyperref[Theorem A]{Theorem A}, \hyperref[Theorem B]{Theorem B} and \hyperref[Theorem C]{Theorem C}. Finally, the last section proves  \hyperref[Theorem D]{Theorem D} and \hyperref[Theorem E]{Theorem E}.
\section{preliminaries}
Throughout the paper we assume that $p>0$ is a prime number and $R$ is a regular Noetherian ring of characteristic $p>0$.  The  $e$-th Frobenius  map $f^e:R\rightarrow R$ given $f(a)=a^{p^e}$ for all $a\in R$ is a homomorphism of rings. Let $R^e$ denote the $(R,R)$-bimodule which is $(R,+)$
as an additive group with left multiplication being the usual multiplication and right multiplication is via the Frobenius map $f^e$.
\begin{definition}
     The $e$-th Frobenius functor $\mathcal{F}^e:\Mod(R)\rightarrow \Mod(R)$ is defined as follows:$$\mathcal{F}^e(M)= R^e\otimes_R M$$ and $$\mathcal{F}^e(M\xrightarrow{f} N)=id_{R^e}\otimes_R f.$$
\end{definition}
We $e=1$, we denote $\mathcal{F}^e$ by just $\F$. Now we define $\F$-mdoule.
\begin{definition}
     An $\F_R$-module or  $\F$-module over $R$ (or simply an $\F$-module, if this causes no confusion) is an $R$-module $M$ equipped with an $R$-module isomorphism $\theta: M \rightarrow \F(M)$ which we call the structure morphism of $M$.
\end{definition}
If $M$ is a $\mathbb{Z}^n$-graded $R$-module then there is a natural $\mathbb{Z}^n$-grading on $\F_R(M)$ defined by $\deg(r'\otimes m)=\deg(r')+p\ \deg(m)$ for homogeneous elements $r'\in R$ and $m\in M$. An $\F$-module $(M,\theta)$ is called $\mathbb{Z}^n$-graded $\F$-module if $M$ is $\mathbb{Z}^n$-graded and $\theta$ is degree preserving.

Given any finitely generated $R$-module $U$ and a $R$-linear map $\beta: U \rightarrow \F(U)$ one can obtain an $R$-module

$$
M=\lim _{\longrightarrow}\left(U \xrightarrow{\beta} \F(U) \xrightarrow{\F(\beta)} \F^2(U) \xrightarrow{\F^2(\beta)} \cdots\right).
$$
Since
$$
\begin{aligned}
\mathcal{F}(M) &= \lim_{\longrightarrow}\left(\F(U) \xrightarrow{F(\beta)} \F^2(U) \xrightarrow{\F^2(\beta)} \F^3(U) \xrightarrow{\F^3(\beta)} \cdots\right) \\
     &= M
\end{aligned}$$

therefore $M \cong \mathcal{F}(M)$, and hence $M$ is an $\F$-module. 
\begin{definition}
     Any $\F$-module which can be constructed as a direct limit as $M$ above is called an $\F$-finite, $\F$-module with generating morphism $\beta$.
\end{definition}
If $M$ is $\mathbb{Z}^n$-graded and $\beta$ is degree preserving, we say that $M$ is a $\mathbb{Z}^n$-graded $\F$-finite, $\F$-module over $R$. The primary examples of $\F$-finite, $\F$-module over $R$ are composition of local cohomology modules of the form $ H^{i_1}_{I_1}(H^{i_2}_{I_2}(\dots H^{i_r}_{I_r}(R)\dots)$. In fact, if $M$ is $\F$-finite, $\F$-module over $R$, so is the local cohomology modules $H^i_I(M)$ for each ideal $I$ of $R$ and for each $i\geq 0$; see \cite[Proposition 2.10]{Lyu-97}. Let us recall some results that we need.
\begin{theorem}\cite[Theorem 1.1]{TP-koszul}\label{F-finite}
        Let $K$ be an infinite field of characteristic $p>0$. Let $R$ be one of the following rings 
        \begin{enumerate}[\rm (i)]
            \item $K[Y_1,\ldots,Y_n]$
            \item  $K[[Y_1,\ldots,Y_n]]$
            \item  $A[X_1,\ldots,X_m]$ where $A= K[[Y_1,\ldots,Y_n]].$
        \end{enumerate}
        Let $M$ be a $\F$-finite, $\F$-module module over $R$. Fix $r\geq 1$. Then, the Koszul homology modules $H_i(Y_1,\ldots,Y_r;M)$ are $\F$-finite, $\F$-modules over $\overline{R}$ where $\overline{R}=R/(Y_1,\ldots,Y_r)$ and for $i=0,\ldots,r$.
    \end{theorem}
    Even if $K$ is a finite field, $M$  a $\F$-finite, $\F$-module over $R$, we have $M/ZM$ is $\F$-finite, $\F$-module over $R/ZR$ if $Z\in R$ is a regular element. This follows from the following result.
    \begin{theorem}\cite[Proposition 2.9]{Lyu-97}\label{pi take f finite to f finite}
        Let $\pi : R\rightarrow B$ be a homomorphism of rings where $B$ is regular. If $M$ is a $\F$-finite, $\F$-module over $R$, then $B\otimes_R M$ is $\F$-finite, $\F$-module over $B$.
    \end{theorem}
    
An analogous proof extends to the $\mathbb{Z}^n$-graded setting of Lemma 8.1 from \cite{TP-koszul} as follows:
\begin{theorem}\label{ordinal}
Let $A=K[[Z_1,\ldots,Z_d]]$ where $K$ is an infinite field of characteristic $p>0$. Let $M=\bigoplus_{\underline{u}\in \mathbb{Z}^n} M_{\underline{u}}$ be a $\mathbb{Z}^n$-graded $\F$-finite, $\F$-module over $R$  where $R=A[X_1,\ldots, X_n]$ is standard $\mathbb{N}^n$-graded polynomial ring over $A$, i.e., $\deg (A)=\underline{0}\in \mathbb{N}^n$ and $\deg(X_j)=e_j\in \mathbb{N}^n$. If $M_{\underline{u}}$ is supported only at maximal ideal of $A$, then 
\begin{enumerate}[\rm (i)]
    \item $M_{\underline{u}}\cong E^{\alpha}$ where $E$ is injective hull of $K$ as $A$-module and $\alpha$ is an ordinal which may be infinite.
    \item If $M_{\underline{u}}\cong E_A(K)^{\alpha}$, then  $H_d(Z_1,\ldots,Z_d,M)_{\underline{u}}\cong K^\alpha$.
\end{enumerate}
\end{theorem}
The following result is well-known.
\begin{proposition}\label{ass-contraction}
    If $f: A\rightarrow B$ is a homomorphism of Noetherian rings and $M$ is a $B$-module, then the associated primes of $M$ as an $A$-module are contractions of the associated primes as a $B$-module to $A$. More precisely, $\Ass_AM=A\cap \Ass_BM$.
\end{proposition}
\begin{lemma}\cite[Lemma 4.3]{ST-25}
\label{p does not divide}
    If $p$ is a prime number and $t\in\mathbb{Z}$ with $p\nmid t$, then $p\nmid \left[\binom{P^et-1}{p^e}+1\right]$.
\end{lemma}
 Let $R=K[X_1,\ldots, X_n]$ be a standard $\mathbb{N}^n$-graded polynomial ring over $A$, i.e., $\deg (K)=\underline{0}\in \mathbb{N}^n$ and $\deg(X_j)=e_j\in \mathbb{N}^n$. Now we recall the definition of straight module over $R$ which was introduced by Yanagawa in \cite{Yanagawa-01}.
\begin{definition}
    A $\mathbb{Z}^n$-graded $R$-module $M=\bigoplus_{\underline{u}\in \mathbb{Z}^n}M_{\underline{u}}$ is called straight, if the following two conditions are satisfied.
    \begin{enumerate}[\rm (a)]
        \item$ \dim_K M_{\underline{u}}<\infty$ for all $\underline{u}\in \mathbb{Z}^n$.
        \item The multiplication map $M_{\underline{u}}\ni Y\rightarrow \underline{X}^{\underline{v}} Y\in M_{\underline{u}+\underline{v}} $ is bijective for all $\underline{u}\in \mathbb{Z}^n$ and  $\underline{v}\in \mathbb{N}^n$ with $\operatorname{supp}_{+}(\underline{u}+\underline{v})= \operatorname{supp}_{+}(\underline{u})$.
    \end{enumerate}
\end{definition}
For a $\mathbb{Z}^n$-graded module $M=\bigoplus_{\underline{u}\in \mathbb{Z}^n}M_{\underline{u}}$, let $M(\underline{-1})$ denote the module $M(-1,\ldots,-1)$. The following fact will be useful for us.
\begin{lemma}\cite[Lemma 3.2]{ST-26}\label{iso on each block}
Let    $M(\underline{-1})$ be a straight module over $K[X_1,\ldots,X_n]$. Let $U$ be a subset of $\{1,\ldots,n\}$. Then, $M_{\underline{u}}\cong M_{\underline{v}}$ for $u,v\in \mathcal{B}(U)$.
\end{lemma}
\section{Multigraded Eulerian module}
The notion of  $\Z$-graded Eulerian $\D$-module was defined by Ma and Zhang in \cite{MaZh-14} for standard graded polynomial ring $R=K[X_1,\ldots,X_n]$ over a field $K$. Here $\D$ is the corresponding ring of $K$-linear differential operator on $R$. 

 Suppose that $R=K[X_1,\ldots,X_n]$ is a standard $\mathbb{N}^n$-graded polynomial ring over $K$, i.e., $\deg (K)=\underline{0}\in \mathbb{N}^n$ and $\deg(X_j)=e_j\in \mathbb{N}^n$. The corresponding ring of differential operators $\D$ of $R$ is naturally $\mathbb{Z}^n$-graded where $\deg(X_i)=e_i$ for all $i$, $\deg (\partial_i^{[j]})=-e_j$ for all $j$, and $\deg (a)=\underline{0}$ for all nonzero $a\in K$.
Under this setup, as a natural extension of $\mathbb{Z}$-graded Eulerian $\D$-module, we define $\mathbb{Z}^n$-graded Eulerian $\D$-module.
\begin{definition}
    For $i,r\geq 1$, define $E_i^r :=X_i^r\  \partial_i^{[r]}$. A $\mathbb{Z}^n$-graded $\D$-module $M$ is called a Eulerian if for every $i\geq 1$, each  element $y\in M$ with $\deg y=\underline{u}$ satisfies
    $$E^r_i\cdot y=\binom{u_i}{r}\cdot y$$
  for all $r\geq 1$.
\end{definition}
The authors in \cite{ST-26} have shown that the polynomial ring $R$ \cite[Lemma 2.4]{ST-26} as well as the local cohomology modules of the form $H^i_I(R)$ \cite[Remark 2.6]{ST-26} are $\mathbb{Z}^n$-graded Eulerian $\D$-module.
\begin{remark}\normalfont
If $(M,\theta)$ is an $\F$-module then the map $$\alpha_e:M\xrightarrow{\theta}\F(M)\xrightarrow{F(\theta)}\F^2(M)\xrightarrow{F^2(\theta)}...\rightarrow \F^e(M)$$
 induced by $\theta$ is also an isomorphism.  We show that $M$ is a $\D$-module. It is enough to specify the action of $\partial_i^{[r]}$ on $M$. Choose $e$ such that $p^e\geq r+1$. Given an element $m$, we have $\alpha_e(m)=\sum_j y_j\otimes z_j $ where $y_j\in$ $R^e$ and $z_j\in M$ and we define $$\partial_i^{[r]} m:=\alpha_e^{-1}\left(\sum_j \partial_i^{[r]} y_j\otimes z_j\right).$$
 \end{remark} 

The following result is crucial for us.
\begin{theorem}
\label{d-mod}
    If $M$ is $\mathbb{Z}^n$-graded $\F$-module, then $M$ is $\mathbb{Z}^n$-graded Eulerian $\D$-module.
\end{theorem}
\begin{proof}
    Let $m\in M$ be any homogeneous element of degree $(m_1,\ldots,m_n)$. Pick $e$ such that $p^e\geq r+1$. Since $M$ is a $\mathbb{Z}^n$-graded $\F$-module, we have a degree preserving isomorphism $\alpha_e: M\rightarrow \F^e_R(M)$. Let $\alpha_e(m)=\sum_j y_j\otimes z_j $ where $y_j\in R$ and $z_j\in M$ are homogeneous. Now $(m_1,\ldots,m_n)=\deg(m)=\deg(\alpha_e(m))=\deg(y_j\otimes z_j)=p^e\deg(z_j)+\deg(y_j)=p^e(a^j_1,\ldots,a_n^j)+(b^j_1,\ldots,b_n^j)=(p^ea^j_1+b^j_1,\ldots,p^ea^j_n+b^j_n)$, where $\deg(z_j)=(a^j_1,\ldots,a_n^j)$ and $\deg(y_j)=(b^j_1,\ldots,b_n^j)$. Therefore $m_i=p^ea^j_i+b^j_i$. Now 
   \begin{align*}
       E^r_im &=\alpha_e^{-1}\left(\sum_j(E^r_i y_j)\otimes z_j\right)\\&=\alpha_e^{-1}\left(\sum_j\binom{b^j_i}{r} y_j\otimes z_j\right)\\&=\alpha_e^{-1}\left(\binom{b^j_i}{r}\sum_j y_j\otimes z_j\right)\\&=\binom{b^j_i}{r} m\\&=\binom{p^ea^j_i+b^j_i}{r}m\\&=\binom{m_i}{r} m.
   \end{align*} 
  This proves the Theorem.
\end{proof}
\section{Vanishing, Rigidity and Straight modules}\label{vanishing , rigidity over K section}
In this section, we prove  vanishing and rigidity results for graded components of a $\mathbb{Z}^n$-graded $\F$-finite, $\F$-module over $R$ under certain assumptions.
\begin{proposition}\label{koszul concentrated in degree 0}
   Let $K$ be a  field of characteristic $p>0$. Suppose that $R=K[X_1,\ldots,X_n]$ is a standard $\mathbb{N}^n$-graded polynomial ring over $K$, i.e., $\deg (K)=\underline{0}\in \mathbb{N}^n$ and $\deg(X_j)=e_j\in \mathbb{N}^n$. Assume that $M=\bigoplus_{\underline{u}\in \mathbb{Z}^n} M_{\underline{u}}$ is a $\mathbb{Z}^n$-graded $\F$-finite, $\F$-module over $R$. Then, for all $i
   \geq 1$, $H_j(X_i,M)_{\underline{u}}=0$ for $u_i\neq 0$ and for $j=0,1$. 
\end{proposition}
\begin{proof}
    Let $u_i\neq 0$ and  $m\in M_{\underline{u}}$. Assume $u_i=p^et$ where $p$ does not divide $t$. Since $M$ is $\mathbb{Z}^n$-graded Eulerian $$E_i^{p^e}m=X_i^{p^e}\partial_i^{[p^e]}m=\binom{p^et}{p^e}m.$$
    It can be easily verified that $\displaystyle \binom{p^et}{p^e}$ is not divisible by $p$. It follows that $m\in X_iM$. Thus, we get $H_0(X_i,M)_{\underline{u}}=0$.
    
    Now let $m\in H_1(X_i,M)_{\underline{u}}\subseteq M_{\underline{u}-e_i}$. Since $M$ is $\mathbb{Z}^n$-graded Eulerian  so
    \begin{align*}
        X_i\partial_i \ m=(u_i-1)m &\implies (\partial_i X_i-1)m=(u_i-1)m\\& \implies u_im=0\\&\implies m=0 \ \text{if}\  p\nmid u_i.
    \end{align*}
     If $p\mid u_i$, then let $u_i=p^et$ where $p\nmid t$. It is clear that $p\mid \binom{p^e}{i}$ for $i\neq 0$ and $i\neq p^e$. We have the following formula \cite[Lemma 2.3]{MaZh-14};
     \begin{align*}  
     \partial_i^{[p^e]}X_i^{p^e}&=\sum_{i=0}^{p^e}\binom{p^e}{i}X_i^{p^e-i}\partial_i^{[p^e-i]}\\&=X_i^{p^e}\partial_i^{[p^e]}+1.
     \end{align*}
Therefore,
$$0=\partial_i^{[p^e]}X_i^{p^e}m=X_i^{p^e}\partial_i^{[p^e]}m+m.$$
Since $M$ is $\mathbb{Z}^n$-graded Eulerian  and $m\in M_{\underline{u}}$ we get
$$\left[\binom{p^et-1}{p^e}+1\right]m=0\implies m=0.$$ since $p\nmid \left[\binom{p^et-1}{p^e}+1\right]$ by Lemma \ref{p does not divide}.
\end{proof}
The next result  proves that the vanishing of almost graded components of $M$ implies vanishing of $M$. More precisely we prove the following.
\begin{theorem}\label{Vanishing of M}
     Let $K$ be an infinite  field of characteristic $p>0$. Suppose that $R=K[X_1,\ldots,X_n]$ is a standard $\mathbb{N}^n$-graded polynomial ring over $K$, i.e., $\deg (K)=\underline{0}\in \mathbb{N}^n$ and $\deg(X_j)=e_j\in \mathbb{N}^n$. Assume that $M=\bigoplus_{\underline{u}\in \mathbb{Z}^n} M_{\underline{u}}$ is a $\mathbb{Z}^n$-graded $\F$-finite, $\F$-module over $R$. If $M_{\underline{u}}=0$ for all $|u_i|\gg 0 $, then $M=0$.
\end{theorem}
\begin{proof}
    The result when $n=1,2$ follows from  \cite[Theorem 1.5]{TP-koszul} and  \cite[Theorem 4.8]{ST-25} respectively. Let $n\geq 2$ and assume that the result is true for all values less than or equal to $n-1$. Now by Theorem \ref{F-finite}, $H_j(X_i,M)$ is $\F$-finite, $\F$-module over $K[X_1,\ldots,\widehat{X_i},\ldots,X_n]$ for $j=0,1$. Also since $M_{\underline{u}}=0$ for all $|u_i|\gg 0 $, we get that $H_j(X_i,M)_{\underline{u}}=0$ for all $|u_i|\gg 0 $. Therefore by induction hypothesis $H_j(X_i,M)=0$ for $j=0,1$.
    Therefore, the following short exact sequence $$0\rightarrow H_1(X_i,M)\rightarrow M(-e_i)\xrightarrow{X_i}M\rightarrow H_0(X_i,M)\rightarrow 0$$ implies that $M(-e_i)\xrightarrow{X_i}M$ i.e.,  $M_{\underline{u}-e_i}\xrightarrow{X_i}M_{\underline{u}}$ is an isomorphism for all $i\geq 1$. Now for a fixed $\underline{u}\in \mathbb{Z}^n$, choose a sufficiently large $t_{\underline{u}}\in \mathbb{N}$ such that $M_{\underline{u}-t_{\underline{u}}\sum_{i=1}^ne_i}=0$. Since, $M_{\underline{u}-t_{\underline{u}}\sum_{i=1}^ne_i}\xrightarrow{\left(\prod_{i=1}^{n} X_i\right)}^{t_{\underline{u}} }M_{\underline{u}}$ is an isomorphism, we get $M_{\underline{u}}=0$.
\end{proof}
 Let $\mathcal{S}$ denote the set $\{1, \ldots, n\}$ and $U$ be a subset (may be empty) of $\mathcal{S}$.  We define a block to be
$\mathcal{B}(U)=\{\underline{u} \in \Z^n \mid u_i \geq 0 \mbox{ if } i \in U \mbox{ and } u_i \leq -1 \mbox{ if } i \notin U \}.$ Note that $\bigcup_{U\subseteq \mathcal{S}}\mathcal{B}(U)=\mathbb{Z}^n$. In the next result, we study the rigidity property of the components of $\mathbb{Z}^n$-graded $\F$-finite, $\F$-modules over $K[X_1,\ldots,X_n]$. 
\begin{theorem}\label{rigidity theorem}
     Let $K$ be a  field of characteristic $p>0$. Suppose that $R=K[X_1,\ldots,X_n]$ is a standard $\mathbb{N}^n$-graded polynomial ring over $K$, i.e., $\deg (K)=\underline{0}\in \mathbb{N}^n$ and $\deg(X_j)=e_j\in \mathbb{N}^n$. Assume that $M=\bigoplus_{\underline{u}\in \mathbb{Z}^n} M_{\underline{u}}$ is a $\mathbb{Z}^n$-graded $\F$-finite, $\F$-module over $R$. Then, the following are equivalent:
     \begin{enumerate}[\rm (i)]
         \item $M_{\underline{u}}\neq 0$ for all $\underline{u}\in \mathcal{B}(U)$;
         \item $M_{\underline{w}}\neq 0$ for some $\underline{w}\in \mathcal{B}(U)$.
     \end{enumerate}
\end{theorem}
\begin{proof}
We only need to prove (ii) implies (i). 

Consider the following short exact sequence $$0\rightarrow H_1(X_i,M)\rightarrow M(-e_i)\xrightarrow{X_i}M\rightarrow H_0(X_i,M)\rightarrow 0.$$ 
By Proposition \ref{koszul concentrated in degree 0}, we see that for all $i\geq 1$, $M(-e_i)_{\underline{u}}\xrightarrow{X_i}M_{\underline{u}}$ is an isomorphism, i.e.,  $M_{\underline{u}}\cong M_{\underline{u}-e_i}$ is an isomorphism if $u_i\neq 0$. 

 Let $U=\{s_1,s_2,\ldots,s_t\}\subseteq \mathcal{S}$. 
    Let $\underline{u}$ and $\underline{v}$ be two elements of $\mathcal{B}(U)$. Let $1\leq j\leq t$. By definition, we have $u_{s_j}\geq 0$. Therefore,
    \begin{align*}
        M_{(u_1,\ldots,0,\ldots,u_n)}\ &\cong M_{(u_1,\ldots,1,\ldots,u_n)}\\&\cong  M_{(u_1,\ldots,2,\ldots,u_n)}\\& \multicolumn{2}{c}{\vdots}  \\&\cong M_{(u_1,\ldots,u_{s_j},\ldots,u_n)} \\&\cong M_{(u_1,\ldots,u_{s_j}+1,\ldots,u_n)}\\& \multicolumn{2}{c}{\vdots}
    \end{align*}
     Hence,
$M_{(u_1,\ldots,u_{s_j},\ldots,u_n)}\cong M_{(u_1,\ldots,u_{s_j}+g,\ldots,u_n)}$ for all $g\geq - u_{s_j}$. Let $q$ be such that $q\neq s_i$ for any $1\leq i\leq t$. Then a similar argument proves that $M_{(u_1,\ldots,u_q,\ldots,u_n)}\cong M_{(u_1,\ldots,u_{q}-g,\ldots,u_n)}$ for all $g\geq u_q+1$. We note that $v_{s_j}=u_{s_j}+a$ for some $a\geq - u_{s_j}$. Also if $q$ is such that $q\neq s_i$ for any $1\leq i\leq t$ then $v_q=u_q-b$ for some $b\geq u_q+1$. Consequently, $M_{\underline{u}}\cong M_{\underline{v}}$. This proves the result.
\end{proof}
\begin{theorem}\label{each component have finite dimension}
     Let $K$ be a  field of characteristic $p>0$. Suppose that $R=K[X_1,\ldots,X_n]$ is a standard $\mathbb{N}^n$-graded polynomial ring over $K$, i.e., $\deg (K)=\underline{0}\in \mathbb{N}^n$ and $\deg(X_j)=e_j\in \mathbb{N}^n$. Assume that $M=\bigoplus_{\underline{u}\in \mathbb{Z}^n} M_{\underline{u}}$ is a $\mathbb{Z}^n$-graded $\F$-finite, $\F$-module over $R$. Then, $\dim_K(M_{\underline{u}})<\infty$ for all $\underline{u}\in\mathbb{Z}^n$.
\end{theorem}
\begin{proof}
    We may assume that $K$ is an algebraically closed field and hence $[K:K^p]<\infty$. Let $\D$ be the corresponding ring of  $K$-linear differential operators on $R$. Note that $\D_{\underline{0}}=K\left\langle X_i^r\partial_i^{[r]}\mid 1\leq i\leq n,r\geq 1\right\rangle$.  Since $M$ is $\F$-finite, $\F$-module so by \cite[Corollary 5.8]{Lyu-97}, $M$ has finite length in the category of $\D$-module. We claim that $M_{\underline{u}}$ is Noetherian as $\D_{\underline{0}}$-module for any $\underline{u}\in\mathbb{Z}^n$. Indeed if $T$ is a $\D_{\underline{0}}$-submodule of $M_{\underline{u}}$ then $\D T\cap M_{\underline{u}}=T$. If $$T_1\subseteq T_2\subseteq\ldots\subseteq T_r\subseteq T_{(r+1)}\subseteq\ldots$$ is an ascending chain of $\D_{\underline{0}}$-submodule of $M_{\underline{u}}$, then we have ascending chain of $\D$-submodules of $M$ 
$$\D T_1\subseteq \D T_2\subseteq\ldots\subseteq \D T_r\subseteq \D T_{r+1}\subseteq\ldots.$$ Since $M$ is Noetherian there exists $t$ such that $\D T_r=\D T_t$ for all $r\geq t$. Intersecting with $M_{\underline{u}}$, we get $T_r=T_t$ for all $r\geq t$. Hence $M_{\underline{u}}$ is Noetherian as $\D_{\underline{0}}$-module and therefore a finitely generated $\D_{\underline{0}}$-module. Let $m_1,\ldots,m_s$ be the finite generators of $M_{\underline{u}}$ and let $\deg m_i=(m_1^i,m_2^i,\ldots,m_n^i)$ for all $1\leq i\leq s$. Since $M$ is $\mathbb{Z}^n$-graded Eulerian by Theorem \ref{d-mod}, we have
$$X_i^r\partial_i^{[r]} m_j=\binom{m^j_i}{r}m_j$$ for all $r\geq 1$ and for all $1\leq i\leq n$. This implies $\D_{\underline{0}}M_{\underline{u}}\subseteq K m_1+\ldots+Km_s$. Hence $\dim_K(M_{\underline{u}})<\infty$ for all $\underline{u}\in\mathbb{Z}^n$.
\end{proof}
It is proved \cite[Remark 2.13]{Yanagawa-01} that if $M$ is straight module over $R=K[X_1,\ldots,X_n]$, then $M$ is $\F$-finite, $\F$-module over $R$. In the next result, we establish the converse when $\operatorname{char}(K)=p>0$.
\begin{lemma}\label{M(-1) is staright}
     Let $K$ be a  field of characteristic $p>0$. Suppose that $R=K[X_1,\ldots,X_n]$ is a standard $\mathbb{N}^n$-graded polynomial ring over $K$, i.e., $\deg (K)=\underline{0}\in \mathbb{N}^n$ and $\deg(X_j)=e_j\in \mathbb{N}^n$. Assume that $M=\bigoplus_{\underline{u}\in \mathbb{Z}^n} M_{\underline{u}}$ is a $\mathbb{Z}^n$-graded $\F$-finite, $\F$-module over $R$. Then, $M(\underline{-1})$ is a straight module over $K[X_1,\ldots,X_n]$.
\end{lemma}
\begin{proof}
Note that $M(\underline{-1})_{\underline{u}}\xrightarrow{X_i}M(\underline{-1})_{\underline{u}+e_i}$ is an isomorphism when $\operatorname{supp}_+(\underline{u})=\operatorname{supp}_+(\underline{u}+e_i)$ is equivalent to saying that $M_{\underline{u}}\xrightarrow{X_i} M_{\underline{u}+e_i}$ is an isomorphism when $\operatorname{supp}_{+}(\underline{u}+e_1+\ldots+e_n)= \operatorname{supp}_{+}(\underline{u}+e_1+\ldots+2e_i+\ldots+e_n)$. 

Let $\underline{u}\in \mathbb{Z}^n$ with $\operatorname{supp}_{+}(\underline{u}+e_1+\ldots+e_n)= \operatorname{supp}_{+}(\underline{u}+e_1+\ldots+2e_i+\ldots+e_n)$. This implies that either $u_i>-1$ or $u_i\leq -2$. Therefore, $u_i+1\neq 0$. Let $\underline{v}\in \mathbb{Z}^n$ with $v_i=u_i+1$ and $v_j=u_j$ for $j\neq i$. Consider the following short exact sequence $$0\rightarrow H_1(X_i,M)\rightarrow M(-e_i)\xrightarrow{X_i}M\rightarrow H_0(X_i,M)\rightarrow 0.$$ 
By Proposition \ref{koszul concentrated in degree 0}, we see that $M(-e_i)_{\underline{v}}\xrightarrow{X_i}M_{\underline{v}}$ i.e.,  $M_{\underline{v}-e_i}\xrightarrow{X_i}M_{\underline{v}}$ is an isomorphism. Therefore, $M_{\underline{u}}\xrightarrow{X_i} M_{\underline{u}+e_i}$ is an isomorphism. Also by Theorem \ref{each component have finite dimension}, $\dim_K(M_{\underline{u}})<\infty$ for all $\underline{u}\in\mathbb{Z}^n$. This proves that $M(\underline{-1})$ is a straight module over $K[X_1,\ldots,X_n]$.
\end{proof}
Let $\omega_R$ denote the caonical module of $R$. By \cite[Example 3.6.10]{BH-93}, we have that $\omega_R\cong R(\underline{-1})$. Lemma \ref{M(-1) is staright} gives an alternative proof a result due Musta\c t\v a \cite{Mustata-20} and Terai \cite{Terai-98} in positive characteristic.
\begin{theorem}\cite{Mustata-20, Terai-98}\label{theorem of terai and mustata}
   Let  $I_{\scriptscriptstyle \triangle}$
 be a squarefree monomial ideal of $R=K[X_1,\ldots,X_n]$ where $K$ is a field of characteristic $p>0$. For all $i\geq 0$, the local cohomology module $H^i_{I_{\scriptscriptstyle \triangle}}(\omega_R)=H^i_{I_{\scriptscriptstyle \triangle}}(R)(-1,\ldots,-1)$ is a straight module.
 \end{theorem}
\section{Finiteness of some invariants related to the graded components}
In this section, we extend the results established in Section \ref{vanishing , rigidity over K section} to a more general setup. More precisely, the setup for this section is as follows.
\begin{setup}\label{main setup}
    Let $A$ be a regular ring containing a field of characteristic $p>0$. Suppose that $R=A[X_1,\ldots,X_n]$ is a standard $\mathbb{N}^n$-graded polynomial ring over $A$, i.e., $\deg (A)=\underline{0}\in \mathbb{N}^n$ and $\deg(X_j)=e_j\in \mathbb{N}^n$. Assume that $M=\bigoplus_{\underline{u}\in \mathbb{Z}^n} M_{\underline{u}}$ is a $\mathbb{Z}^n$-graded $\F$-finite, $\F$-module over $R$.
\end{setup}
\begin{remark}\normalfont
\label{stan}
  Let the hypothesis be as in Setup \ref{main setup}.  Suppose $M_{\underline{u}}\neq 0$ for some $\underline{u}\in \mathbb{Z}^n$. Let $\p$ be a minimal prime ideal of $M_{\underline{u}}$ and let $B=\widehat{A_\p}$. Set $S=B[X_1,\ldots,x_n]$. By Theorem \ref{pi take f finite to f finite}, $N=S\otimes_RM=B\otimes_A M$ is $\F$-finite, $\F$-module over $S$. By Cohen structure theorem $B=K[[t_1,\ldots,t_g]]$ where $K=\kappa(\p)$ is the residue field of $B$. We note that $N_{\underline{u}}={(M_{\underline{u}})_\p}\neq 0$. Hence $N_{\underline{u}}$ is supported only at the maximal ideal of $B$ as $\p$ is minimal prime ideal of $M_{\underline{u}}$. Therefore by Theorem \ref{ordinal}, $N_{\underline{u}}=E_B(K)^{\alpha}$ for some $\alpha$. 

    If $K$ is finite, we consider an infinite field $K'$ containing $K$ and consider the flat extension $B\rightarrow C=K'[[t_1,\ldots,t_g]]$. Set $T=C[X_1,\ldots,X_n]$, a flat extension of $S$. By Theorem \ref{pi take f finite to f finite}, $L=T\otimes_S N=C\otimes_B N$ is $\F$-finite, $\F$-module over $T$. Therefore $L_{\underline{u}}\neq 0$ and supported only at the maximal ideal of $C$. We note that $L_{\underline{u}}=E_C(K')^{\alpha}$.

    Let $V=H_g(t_1,\ldots,t_g;L)$. Then,  $V$ is  $\F$-finite, $\F$-module over $D=K'[X_1,\ldots,X_n]$ and by Theorem \ref{ordinal} $V_{\underline{u}}\cong (K')^\alpha$. Also $V\subseteq L$ and $L=M\otimes_A C$.
\end{remark}
The above remark is useful in extending the results proved in Section (\ref{vanishing , rigidity over K section}) to a more general setup, more precisely when $R$ is as in Setup \ref{main setup}.
\begin{theorem}\label{proof of vanishing theorem}
    Let $R$ be as in Setup \ref{main setup}. Assume that $M=\bigoplus_{\underline{u}\in \mathbb{Z}^n} M_{\underline{u}}$ is a $\mathbb{Z}^n$-graded $\F$-finite, $\F$-module over $R$. If $M_{\underline{u}}=0$ for all $|u_i|\gg 0 $, then $M=0$.
\end{theorem}
\begin{proof}
    Suppose if possible $M_{\underline{v}}\neq 0$ for some $\underline{v}\in \mathbb{Z}^n$. Let $L$ be as in the Remark \ref{stan}. Now $V_{\underline{u}}=0$ for all $|u_i|\gg 0$  since $V\subseteq L$. So by Theorem \ref{Vanishing of M}, $V=0$ since $V$ is $\F$-finite, $\F$-module over $D=K'[X_1,\ldots,X_n]$. But this is a  contradiction to the fact that $V_{\underline{u}}\neq 0$.
\end{proof}
Next, we prove rigidity results concerning the graded components of $M$.
\begin{theorem}\label{proof of rigidity `t`heorem}
    Let $R$ be as in Setup \ref{main setup}. Assume that $M=\bigoplus_{\underline{u}\in \mathbb{Z}^n} M_{\underline{u}}$ is a $\mathbb{Z}^n$-graded $\F$-finite, $\F$-module over $R$. Then, the following are equivalent:
     \begin{enumerate}[\rm (i)]
         \item $M_{\underline{u}}\neq 0$ for all $\underline{u}\in \mathcal{B}(U)$;
         \item $M_{\underline{w}}\neq 0$ for some $\underline{w}\in \mathcal{B}(U)$.
     \end{enumerate}
\end{theorem}
\begin{proof}
We only need to prove (ii) implies (i). 

     Suppose $M_{\underline{w}}\neq 0$ for some $\underline{w}\in \mathcal{B}(U)$.  Again by Remark \ref{stan}, $L_{\underline{w}}\neq 0$. Now $V$ is $\F$-finite, $\F$-module over $D=K'[X_1,\ldots,X_n]$ with $V_{\underline{w}}\neq 0$. By Theorem \ref{rigidity theorem}, $V_{\underline{u}}\neq 0$ for all $\underline{u}\in \mathcal{B}(U)$. This implies $L_{\underline{u}}\neq 0$ for all $\underline{u}\in \mathcal{B}(U)$. Therefore $M_{\underline{u}}\neq 0$ for all $\underline{u}\in \mathcal{B}(U)$ as $L_{\underline{u}}=M_{\underline{u}}\otimes_A C$.
\end{proof}
We need the following Lemma from \cite[1.4]{Lyu-93}.
\begin{lemma}\label{Lyu result on bass number}
    Let $B$ be a Noetherian ring and let $N$ be a $B$-module. Let $\p$ be a prime ideal in $B$. If $(H^j_\p(N))_\p$ is injective for all $j\geq 0$ then $\mu_j(\p,N)=\mu_0(\p,H^j_\p(N))$ for $j\geq 0$.
\end{lemma}
We now show that the hypothesis of the last stated lemma is satisfied in our case.
\begin{proposition}\label{Injective satisfy} (with hypothesis as in \ref{main setup}) Let $\p$ be a prime ideal of $A$. Set $E=M_{\underline{u}}$ for some $\underline{u}\in \mathbb{Z}^n$. Then $H^j_{\p}(E)_{\p}$ is injective for all $j\geq 0$.
\end{proposition}
\begin{proof}
Since  $M=\bigoplus_{\underline{u}\in \mathbb{Z}^n} M_{\underline{u}}$ is a $\mathbb{Z}^n$-graded $\F$-finite, $\F$-module over $R$, so is $H^j_{\p R}(M)$. Without loss of generality, assume that $H^j_\p(E)_\p\neq 0$. This implies that $\p$ is a minimal prime of $H^j_\p(E)$. Also note that $H^j_\p(E)=H^j_{\p R}(M)_{\underline{u}}$. Applying the technique of Remark \ref{stan} to the  $\F$-finite, $\F$-module $H^j_{\p R}(M)$, we get $N_{\underline{u}}=H^j_\p(E)_\p$ and $N_{\underline{u}}=E_B(K)^{\alpha_{\underline{u}}}$. But note that $E_B(K)=E_A(A/\p)$ as $A$-module which proves the result.
\end{proof}
\begin{theorem}\label{proof of finiteness of bass numbers}
    Let $R$ be as in Setup \ref{main setup}. Assume that $M=\bigoplus_{\underline{u}\in \mathbb{Z}^n} M_{\underline{u}}$ is a $\mathbb{Z}^n$-graded $\F$-finite, $\F$-module over $R$. Fix $\underline{u}\in\mathbb{Z}^n$. Let $\p$ be a prime ideal in $A$. Then for each $j\geq 0$, the Bass number $\mu_j(\p,M_{\underline{u}})$ is finite.
\end{theorem}
\begin{proof}
      By Lemma  \ref{Lyu result on bass number}, and Proposition \ref{Injective satisfy}, $\mu_j(\p,M_{\underline{u}})=\mu_0(\p,H^j_\p(M_{\underline{u}}))$ for all $\underline{u}\in \mathbb{Z}^n$. Assume that $H^j_\p(M_{\underline{u}})\neq 0$. Also either $H^j_\p(M_{\underline{u}})_\p=0$ or $\p$ is a minimal prime of $H^j_\p(M_{\underline{u}})$. Applying the technique of Remark \ref{stan} to the  $\F$-finite, $\F$-module $H^j_{\p R}(M)$, we get that $N_{\underline{u}}=H^j_\p(M_{\underline{u}})_\p$ and $N_{\underline{u}}=E_B(K)^{\alpha_{\underline{u}}}$. Note that $\alpha_{\underline{u}}=\mu_j(\p,M_{\underline{u}})$ and $V=H_g(t_1,\ldots,t_g;N)$ is $\F$-finite, $\F$-module over $D=K[X_1,\ldots,X_n]$ and we may assume that $K$ is infinite. As $V_{\underline{u}}=H_g(t_1,\ldots,t_g;N)_{\underline{u}}=K^{\alpha_{\underline{u}}}$ by Theorem \ref{ordinal}, we get that $\dim_KV_{\underline{u}}=\alpha_{\underline{u}}=\mu_j(P,M_{\underline{u}})$. The result now follows from Theorem \ref{each component have finite dimension}.
\end{proof}
\begin{theorem}\label{injective dimension theorem}
    Let $R$ be as in Setup \ref{main setup}. Assume that $M=\bigoplus_{\underline{u}\in \mathbb{Z}^n} M_{\underline{u}}$ is a $\mathbb{Z}^n$-graded $\F$-finite, $\F$-module over $R$. Fix $\underline{u}\in\mathbb{Z}^n$. Then,
    $\injdim M_{\underline{u}} \leq \dim M_{\underline{u}}$.
\end{theorem}
\begin{proof}
    Let $\p$ be a prime ideal in $A$. Then Lemma \ref{Lyu result on bass number} together with Proposition \ref{Injective satisfy} implies that
	$$
	\mu_j(\p, M_{\underline{u}}) = \mu_0(\p, H^j_\p(M_{\underline{u}})).
	$$
	By Grothendieck's vanishing theorem $H^j_\p(M_{\underline{u}}) = 0$ for all $j > \dim M_{\underline{u}}$, see \cite[6.1.2]{BS-13}.  So $\mu_j(\p, M_{\underline{u}}) = 0$ for all $j > \dim M_{\underline{u}}$. This proves the result.
\end{proof}
In Theorem \ref{rigidity theorem}, we have proved that for a given subset $U$ of $\{1,\ldots,n\}$, the components of $M$ are isomorphic in $\mathcal{B}(U)$. As a consequence, we get the following results.
\begin{theorem}\label{bass numbers,injective dimension, associted primes are constant in block}
Assume the hypothesis as in  \ref{main setup}. Let $\mathcal{B}(U)$ denote a block corresponding to a subset $U$ of $\{1,\ldots,n\}$.
\begin{enumerate}[\rm (i)]
    \item Fix a  prime ideal $\p$ of $A$ and  $i\geq 0$. Then, the set of Bass numbers $\{\mu_i(\p,M_{\underline{u}})\mid \underline{u}\in \mathbb{Z}^n\}$ is constant 
on $\mathcal{B}(U)$.
\item The sets $\{\injdim M_{\underline{u}}\mid \underline{u}\in \mathbb{Z}^n\}$, $\{\dim M_{\underline{u}}\mid \underline{u}\in \mathbb{Z}^n\}$, $\{\Ass_A M_{\underline{u}}\mid \underline{u}\in \mathbb{Z}^n\}$ are constant 
on $\mathcal{B}(U)$
\end{enumerate}
\end{theorem}
Also note that from \cite[Theorem 2.12]{Lyu-97}, $\Ass_R(M)$ is finite  and therefore using Proposition \ref{ass-contraction}, we see that $\bigcup_{\underline{u} \in \mathbb{Z}^n} \Ass_A M_{\underline{u}}   $ is a finite set.
\section{Structure theorem for the graded components}
Now we prove a structure theorem for the graded components $\mathbb{Z}^n$-graded $\F$-finite, $\F$-module over $R$ which is  analogue of the structure theorem for components of $H^i_I(R)$ proved in characteristic zero (see, \cite[Theorem 7.1]{TS-23}). The proof is mostly similar to the proof of \cite[Theorem 7.1]{TS-23}. However it is different in quite a few places. So we are forced to give the whole proof.
\begin{theorem}\label{structure theorem}
    Let $K$ be field of characteristic $p>0$, $A=K[[Y]]$ be a power series ring in one variable and $Q(A)$ be the field of fraction of $A$. Suppose that $R=A[X_1,\ldots,X_n]$ is a standard $\mathbb{N}^n$-graded polynomial ring over $A$, i.e., $\deg (A)=\underline{0}\in \mathbb{N}^n$ and $\deg(X_j)=e_j\in \mathbb{N}^n$. Assume that $M=\bigoplus_{\underline{u}\in \mathbb{Z}^n} M_{\underline{u}}$ is a $\mathbb{Z}^n$-graded $\F$-finite, $\F$-module over $R$. Then,
    $$M_{\underline{u}}\cong E(A/YA)^{a(\underline{u})}\oplus Q(A)^{b(\underline{u})}\oplus A^{c(\underline{u})}$$
    for finite numbers $a(\underline{u}), b(\underline{u}), c(\underline{u})\geq 0$.
\end{theorem}
\begin{proof}
    Let $\m=YA$ and put $N=M_{\underline{u}}$. Then, $\Gamma_\m(N)=\Gamma_{\m R}(M)_{\underline{u}}$. Using Theorem \ref{injective dimension theorem}, $\injdim \Gamma_{\m R}(M)_{\underline{u}}\leq \dim \Gamma_{\m R}(M)_{\underline{u}}=0$. This implies that $\Gamma_{\m }(N)$ is injective and hence $\Gamma_{\m }(N)\cong E(A/YA)^{a(\underline{u})}$. Note that by Theorem \ref{proof of finiteness of bass numbers}, $a(\underline{u})=\mu_0(\m,N)$ is finite. Again since $\Gamma_{\m }(N)$ is injective the following short exact sequence $$0\rightarrow \Gamma_\m(N)\rightarrow N\rightarrow \overline{N}\rightarrow 0$$ splits where $\overline{N}=N/\Gamma_\m(N)$. Thus $N=\Gamma_\m(N)\oplus \overline{N}$. Now we compute $\overline{N}$.

    Let $L=\bigcap_{n=1}^{\infty}Y^n \overline{N}$. Consider the exact sequence $$0\rightarrow (0:_LY)\rightarrow L\xrightarrow{Y}L\rightarrow C\rightarrow 0.$$ We claim that $(0:_LY)=0$. To see that let $a\in (0:_LY)$. This implies $Ya=0$. Let $a=Y^j \overline{b}$ for some $b\in N$. Then $Ya=Y^{j+1}\overline{b}=0$ and hence $Y^{j+1}b\in \Gamma_\m(N)$. Therefore, $Y^{i+j+1}b=0$ for some $i\geq 1$ and this implies $\overline{b}=0$ proving that $a=0$. Now we prove that $C=0$. Let $z\in L$. Note that $L=\bigcap_{n=1}^{\infty}Y^n \overline{N}\subseteq Y \overline{N}$. Hence, $z=Y c$ for some $c\in \overline{N}$. Let $j\geq 1$ be fixed. Then $Yc=z=Y^{j+1}d$ for some $d\in \overline{N}$. Since $Y$ is $\overline{N}$-regular, $c=Y^jd\in Y^j\overline{N}$. This is true for all $j\geq 1$. Consequently, $c\in L$ implying $C=0$. 
    
    Therefore $Y$ acts as an isomorphism on $L$ and hence $L$ is an $Q(A)=A_{Y}$-module. Let $L=Q(A)^{b(\underline{u})}$ for some $b(\underline{u})$. Let $W=Q(A)[X_1,\ldots,X_n]$. Since $M$ is $\F$-finite, $\F$-module over $R$ so $M\otimes_R Q(A)[X_1,\ldots,X_n]=M\otimes_A Q(A)$ is $\F$-finite, $\F$-module over $W$ by Theorem \ref{pi take f finite to f finite}. By Theorem \ref{each component have finite dimension}, $\dim_{Q(A)}\left(M_{\underline{u}}\otimes_A Q(A)\right)<\infty$. Consider the following exact sequence
    \begin{equation*} 
    0\rightarrow \Gamma_Y(N)\rightarrow N\rightarrow \frac{N}{\Gamma_Y(N)}\rightarrow 0 .
\end{equation*}
Since $\Gamma_Y(N)\otimes_A Q(A)=0$, so $\overline{N}\otimes_A Q(A)= N\otimes_A Q(A)$.
Now $L\subseteq \overline{N}$ implies $L\otimes_A Q(A)\subseteq \overline{N}\otimes_A Q(A)= N\otimes_A Q(A)=M_{\underline{u}}\otimes_A Q(A)$. This implies $b(\underline{u})$ is finite.

   Let $\overline{M}=M/\Gamma_{\m R}(M)$. Since $\overline{M}$ is $\F$-finite, $\F$-module over $R$, we get that $H_0(Y,\overline{M})$ is $\F$-finite, $\F$-module over $K[X_1,\ldots,X_n]$ by Theorem \ref{pi take f finite to f finite}. By Theorem $\ref{each component have finite dimension}$, $\dim_K \left(\overline{M}/Y\overline{M}\right)_{\underline{u}}<\infty$. Note that,  $$\dim_K\left(\overline{N}/L\right)/Y\left(\overline{N}/L\right)=\dim_K\left(\overline{M}/Y\overline{M}\right)_{\underline{u}}<\infty.$$ Also $\bigcap_{n=1}^{\infty}Y^n \left(\overline{N}/L\right)=\left(\bigcap_{n=1}^{\infty}Y^n\overline{N}\right)/L=0$. Thus by \cite[Theorem 8.4]{Mat-89}, $\overline{N}/L$ is finitely generated. Consider the following commutative diagram
\begin{equation}\label{equation of gamma M}
\begin{tikzcd}
	0 & L & \overline{N} & \overline{N}/L & 0  \\
	0 & L &  \overline{N}  & \overline{N}/L & 0
	\arrow[from=1-1, to=1-2]
	\arrow[ from=1-2, to=1-3]
    \arrow[ from=2-4, to=2-5]
    \arrow[ "{Y}"', from=1-4, to=2-4]
	\arrow["{Y}"', from=1-2, to=2-2]
	\arrow[ from=1-3, to=1-4]
    \arrow[ from=1-4, to=1-5]
	\arrow["{Y}"', from=1-3, to=2-3]
	\arrow[ from=2-1, to=2-2]
	\arrow[ from=2-2, to=2-3]
	\arrow[ from=2-3, to=2-4]
\end{tikzcd}
\end{equation} 
Using Snake lemma, we get $\overline{N}/L\xrightarrow{Y}\overline{N}/L$ is injective and hence $\overline{N}/L$ is torsion-free. Therefore, using the structure theorem for finitely generated module over the PID $A=K[[Y]]$, we have $\overline{N}/L=A^{c(\underline{u})}$ for some finite $c(\underline{u})\geq 0$. Now the result follows as the following short exact sequence $$0\rightarrow L\rightarrow \overline{N}\rightarrow \overline{N}/L\rightarrow 0 $$ splits. 
   \end{proof}
   In the next theorem, we study the behaviour of the numbers $a(\underline{u}), b(\underline{u}), c(\underline{u})$ appearing appearing in the above structure theorem.
\begin{theorem}\label{numbers are constant}
   Let the hypothesis be as in Theorem \ref{structure theorem}. Then,the sets $\{a(\underline{u})\mid \underline{u}\in \mathbb{Z}^n\}$, $\{b(\underline{u})\mid \underline{u}\in \mathbb{Z}^n\}$ and $\{c(\underline{u})\mid \underline{u}\in \mathbb{Z}^n\}$ are constant 
on $\mathcal{B}(U)$ for each subset $U$ of $\{1, \ldots, n\}$.
\end{theorem}
\begin{proof}
Note that $a(\underline{u})=\mu_0(\m,N)$ where $N=M_{\underline{u}}$. Since, $M$ is a $\mathbb{Z}^n$-graded $\F$-finite, $\F$-module over $R$, by Theorem \ref{bass numbers,injective dimension, associted primes are constant in block}, the set $\{a(\underline{u})\mid \underline{u}\in \mathbb{Z}^n\}$ is constant 
on $\mathcal{B}(U)$ for each subset $U$ of $\{1, \ldots, n\}$. Given that  $$M_{\underline{u}}\cong E(A/YA)^{a(\underline{u})}\oplus Q(A)^{b(\underline{u})}\oplus A^{c(\underline{u})}.$$
   Let $\overline{M}=M/\Gamma_{\m R}(M)$. Then $\overline{M}_{\underline{u}}=Q(A)^{b(\underline{u})}\oplus A^{c(\underline{u})}$. Since $M$ is $\F$-finite, $\F$-module over $R$, so is $\overline{M}$. Therefore, $H_0(Y,\overline{M})$ is $\F$-finite, $\F$-module over $K[X_1,\ldots,X_n]$ by Theorem \ref{pi take f finite to f finite}. This implies  $H_0(Y,\overline{M})(\underline{-1})$ is straight as $K[X_1,\ldots,X_n]$ by Lemma \ref{M(-1) is staright}. Now 
   $$\overline{M}_{\underline{u}}=Q(A)^{b(\underline{u})}\oplus A^{c(\underline{u})}$$ implies $$Y\overline{M}_{\underline{u}}=Q(A)^{b(\underline{u})}\oplus YA^{c(\underline{u})}.$$
   Hence $\overline{M}_{\underline{u}}/Y\overline{M}_{\underline{u}}=(A/YA)^{c(\underline{u})}$. Therefore, the set $\{c(\underline{u})\mid \underline{u}\in \mathbb{Z}^n\}$ is constant 
on $\mathcal{B}(U)$ for each subset $U$ of $\{1, \ldots, n\}$ by Lemma \ref{iso on each block}. Consider the following exact sequence
    \begin{equation*} 
    0\rightarrow \Gamma_Y(M_{\underline{u}})\rightarrow M_{\underline{u}}\rightarrow \frac{M_{\underline{u}}}{\Gamma_Y(M_{\underline{u}})}\rightarrow 0 .
\end{equation*}
Since $\Gamma_Y(M_{\underline{u}})\otimes_A Q(A)=0$, we get that $M_{\underline{u}}\otimes_A Q(A)\cong M_{\underline{u}}/\Gamma_\pi(M_{\underline{u}})\otimes_A Q(A)\cong  Q(A)^{b(\underline{u})+c(\underline{u})}$.
Since $M$ is $\F$-finite, $\F$-module over $R$, we get that $M\otimes_RQ(A)[X_1,\ldots,X_n]=M\otimes_AQ(A)$ is $\F$-finite, $\F$-module over $Q(A)[X_1,\ldots,X_n]$ by Theorem \ref{pi take f finite to f finite}. Therefore, by Lemma \ref{M(-1) is staright} together with Lemma \ref{iso on each block} proves that the set $\{b(\underline{u})+c(\underline{u})\mid \underline{u}\in \mathbb{Z}^n\}$ is constant 
on $\mathcal{B}(U)$ for each subset $U$ of $\{1, \ldots, n\}$. This implies that the set $\{b(\underline{u})\mid \underline{u}\in \mathbb{Z}^n\}$ is constant 
on $\mathcal{B}(U)$ for each subset $U$ of $\{1, \ldots, n\}$. This completes the proof of the theorem.
\end{proof}

\begin{remark}\normalfont
    Let $R$ be as in Theorem \ref{structure theorem}. Note that by Theorem \ref{bass numbers,injective dimension, associted primes are constant in block}, the set of Bass numbers $\{\mu_i(YA,M_{\underline{u}})\mid \underline{u}\in \mathbb{Z}^n\}$ is constant 
on $\mathcal{B}(U)$. In fact, the Bass numbers in this case depend on the numbers appearing in the previous structure theorem.

Since $0\rightarrow A\xrightarrow{Y} A\rightarrow A/YA\rightarrow 0$ is a free resolution of $A/YA$, $\mu_i(YA, M_{\underline{u}})=0$ for $i\geq 2$. 

Now, $\mu_1(YA, M_{\underline{u}})=\dim_{A/YA}\Ext^1_A\left(A/YA,E(A/(Y))^{a(\underline{u})}\oplus Q(A)^{b(\underline{u})}\oplus A^{c(\underline{u})}\right)=c(\underline{u})$, since we have $\Ext^1_A(A/YA,A)\cong A/YA$ and  $\Ext^1_A(A/YA, E(A/YA))=\Ext^1_A(A/YA, Q(A))=0$.

Similarly, $\mu_0(YA, M_{\underline{u}})=\dim_{A/YA}\Hom_A\left(A/YA,E(A/(Y))^{a(\underline{u})}\oplus Q(A)^{b(\underline{u})}\oplus A^{c(\underline{u})}\right)=a(\underline{u})$, since we have  $\Hom_A(A/YA, E(A/YA))\cong A/YA$ and $\Hom_A(A/YA, Q(A))=\Hom_A(A/YA, A)=0$.

Since by Theorem \ref{numbers are constant}, the sets $\{a(\underline{u})\mid \underline{u}\in \mathbb{Z}^n\}$, $\{c(\underline{u})\mid \underline{u}\in \mathbb{Z}^n\}$ are constant on $\mathcal{B}(U)$, we can conclude from here that the set of Bass numbers $\{\mu_i(YA,M_{\underline{u}})\mid \underline{u}\in \mathbb{Z}^n\}$ is constant 
on $\mathcal{B}(U)$.
\end{remark}
We conclude the article with the following two examples.
\begin{example}\normalfont
Let $A=K[[Y]]$ where $\operatorname{char}(K)=p>0$. We mention two examples from \cite{TS-23}. Although the authors in \cite{TS-23} work in characteristic zero, these examples remain valid in positive characteristic as well. For more details, the readers are referred to \cite[Example 8.1, 8.2]{TS-23}.
    \begin{enumerate}[\rm(i)]
        \item  Let $R=A[X]$ and $I=(YX)$ in $R$. Then, $E_A(A/YA)$ is a direct summand of $H^1_I(R)_{\underline{u}}$ and hence of $\Gamma_Y\left(H^1_I(R)_{\underline{u}}\right)$. This provides an example where $a(\underline{u})\neq 0$.
        \item Let $R=A[X_1,X_2]$ and $I=(YX_1,X_2)$. Let $M:=H^2_I(R)$. For this example there is some $\underline{u}$ for which the torsion-free part of $\overline{M_{\underline{u}}}$ is not finitely generated implying  $b(\underline{u})\neq 0$.
    \end{enumerate}
\end{example}

\bigskip

\noindent {\it Acknowledgements:}
I would like to express my sincere gratitude to Prof. Tony J. Puthenpurakal, my PhD supervisor for his careful reading of the manuscript and valuable feedback.
\medskip

\bibliographystyle{plain}

\end{document}